\documentclass{amsart}
\usepackage{amsthm, amscd, amsmath}
\usepackage{amscd}
\usepackage[margin=1.5in]{geometry}
\usepackage{amsmath}
\usepackage{psfrag, graphicx}
\usepackage{pstricks}
\usepackage{epic,eepic}
\usepackage{color}
\usepackage{pstricks-add}
\usepackage{hyperref}
\usepackage{enumitem}

\usepackage{changes}

\usepackage{comment}

\newtheorem{thm}{Theorem}[section]
\newtheorem{definition}[thm]{Definition}
\newtheorem{remark}[thm]{Remark}

\newtheorem{prop}[thm]{Proposition}

\newtheorem{thmx}{Theorem}

\numberwithin{equation}{section}

\def\diag{\mathop{\hbox{\rm diag}}}

\def\id{\mathop{\hbox{\rm id}}}

\def\Diff{\mathrm{Diff}}

\def\dim{\mathop{\hbox{\rm dim}}}
\def\det{\mathop{\hbox{\rm det}}}

\def\ve{\varepsilon}

\def\NN{{\mathbb N}}

\def\RR{{\mathbb R}}

\def\TT{{\mathbb T}}

\def\D2{{\mathbb D}^2}

\def\cD{\mathcal D}
\def\cE{\mathcal E}

\def\cG{\mathcal G}
\def\cH{\mathcal H}

\def\cK{\mathcal K}
\def\cL{\mathcal L}
\def\cM{\mathcal M}
\def\cN{\mathcal N}
\def\cO{\mathcal O}

\def\cQ{\mathcal Q}

\def\cS{\mathcal S}

\def\cU{\mathcal U}

\def\cS{{\mathcal S}}

\def\ds{\displaystyle}
\def\dist{\mathrm{dist}}

\def\beq{\begin{equation}}
\def\eeq{\end{equation}}
\def\beqn{\begin{equation*}}
\def\eeqn{\end{equation*}}

\def\p{\partial}
\def\Int{\mathrm{Int}}
\def\vol{\mathrm{vol}}

\def\wrho{\widetilde{\rho}}
\def\brho{\overline{\rho}}

\def\si{\sigma}

\def\wpsi{\widetilde{\psi}}
\def\wPsi{\widetilde{\Psi}}

\newtheorem{propertyx}{}

\begin{document}
\title[]{
A volume preserving nonuniformly hyperbolic diffeomorphism
with arbitrary number of ergodic components and close to the identity}
\author{Jianyu Chen, Huyi Hu and Yun Yang}

\address{School of Mathematical Sciences, 
Center for Dynamical Systems and Differential Equations, Soochow University, Suzhou, Jiangsu, P.R.China}
\email{jychen@suda.edu.cn}

\address{Department of Mathematics, Michigan State University,  
East Lansing, MI, 48824 United States; School of Mathematical Sciences, Center for Dynamical Systems and Differential Equations, Soochow University, Suzhou, Jiangsu, P.R.China}
\email{hhu@msu.edu; }

\address{Department of Mathematics, Virginia Polytechnic Institute and State University, Blacksburg, VA, 24060 United States}
\email{yunyang@vt.edu}


\begin{abstract}
We prove that  for  any $\ell\in\NN\cup\{\infty\}$ and any $r\in \NN$,
every compact smooth Riemannian manifold $\cM$ of 
$\dim \cM\ge 5$
carries a $C^\infty$ volume  preserving nonuniformly hyperbolic diffeomorphism,
which has exactly $\ell$ ergodic components  (in fact, Bernoulli components) and 
is $C^r$ close to the identity. 
\end{abstract}

\maketitle

\section{Introduction}

\subsection{Background and Main Result}

It is known that the existence of uniform hyperbolicity
yields certain topological restrictions,
for instance, the Euler characteristic 
must be zero for any manifold admitting an Anosov diffeomorphism.
Nevertheless, 
there are no topological obstructions for manifolds to admit
nonuniformly hyperbolic diffeomorphisms.  
The notion of nonuniform hyperbolicity 
was first brought up and thoroughly investigated by Pesin (see~\cite{BP2, BP, Todd-Hassel} for references).
A volume preserving diffeomorphism is said to be
\emph{(completely) nonuniformly hyperbolic}, if 
all of its Lyapunov exponents are nonzero at   almost every point.
On any compact surface, 
Katok constructed in \cite{K}  a Bernoulli  diffeomorphism
that is area preserving and nonuniformly hyperbolic. 
Based on this result,
Brin, Feldman and Katok further proved in \cite{BFK} that 
every compact manifold $\cM$ of  $\dim \cM=m \ge 2$  
carries a Bernoulli diffeomorphism, which has only $2$ non-zero Lyapunov exponents
and $(m-2)$ zero Lyapunov exponents;
meanwhile, Brin proved in \cite{B} that every compact manifold $\cM$ of  $\dim \cM=m \ge 5$  
carries a Bernoulli diffeomorphism with exactly $(m-1)$ non-zero Lyapunov exponents.
To demonstrate the existence of the 
complete nonuniform hyperbolicity in full generality,
Dolgopyat and Pesin later proved in \cite{DP} that 
any compact manifold $\cM$ of $\dim \cM\ge 2$ 
carries a volume preserving 
nonuniformly hyperbolic Bernoulli  diffeomorphism, 
and then a similar result was established 
by Hu, Pesin and Talitskaya in~\cite{HPT}
in the continuous-time case for
$\dim \cM\ge 3$.

The spectral decomposition theorem by Smale~\cite{Smale67} 
asserts that for any Axiom A diffeomorphism,
its non-wandering set  
is a union of finitely many basic sets.
In contrast, 
it was shown by Pesin in \cite{Pesin77} that 
there are at most countable, which might be 
infinitely many, ergodic components for 
a volume preserving nonuniformly hyperbolic diffeomorphism. 
Such example was first 
constructed on the 3-torus by Dolgopyat, Hu and Pesin~\cite{DHP},
in which the diffeomorphism is close to the direct product of an Anosov automorphism  and the identity. 
It was later conjectured by Hu~\cite{H} that 
every smooth compact   manifold $\cM$  
admits a volume preserving nonuniformly hyperbolic diffeomorphism,
which is  close to the identity and possesses countably infinitely many ergodic components.
In this paper, we provide an affirmative answer to
the conjecture given in \cite{H} 
for every compact manifold of dimension 
no less than $5$.

\begin{thmx}
\label{Thma}
Let $\cM$ be a  compact connected smooth Riemannian manifold (possibly with boundary)
of $\dim \cM\ge 5$.
For any $\ve>0$,  any $\ell \in \NN\cup\{\infty\}$ and any $r\in \NN$,
there exists a $C^\infty$ diffeomorphism $F$ of $\cM$
with the following properties:
\begin{enumerate} 
\item $F$ is volume preserving, i.e., 
$F$ preserves the Riemannian volume $\vol_\cM$; 
\item $F$ is nonuniformly hyperbolic, i.e., 
$F$ has non-zero Lyapunov exponents at $\vol_\cM$-almost every point $x\in \cM$; 
\item $F$ has exactly $\ell$ ergodic components 
which are open$\pmod {\vol_\cM}$. In fact, $F$ is Bernoulli on each component;
\item $\dist_{C^r}(F, \id)\le \ve$, i.e.,
$F$ is $\ve$-close to the identity in the $C^r$ topology.
\end{enumerate} 
\end{thmx}

We remark that the dimension restriction 
of Theorem~\ref{Thma} is due to
some limitations of our approach in dimension $2$ (see Remark~\ref{rem: dim2})
and dimension $3, 4$ (see Remark~\ref{rem: dim34}) .
We also emphasize that 
the case when $\ell=\infty$ gives a positive solution to 
the conjecture   in \cite{H}, that is,
the diffeomorphism $F$ has countably infinitely many
ergodic   components.

Another motivation of Theorem~\ref{Thma} is
to explore the prevalence of nonuniform hyperbolicity.
Pesin proposed the following conjecture in~\cite{BP, P, CP}:
\begin{center}
{\it 
$\cH\cap \Diff^{r}(\cM, \vol_\cM)$ is dense in $\Diff^{r}(\cM, \vol_\cM)$ under the $C^r$-topology,
}
\end{center}
where 
$\Diff^{r}(\cM, \vol_\cM)$ is the class of $C^r$ volume preserving diffeomorphisms of $\cM$
for some $r\ge 1$, 
and $\cH$ is the set of  diffeomorphisms of $\cM$ which are 
nonuniformly hyperbolic in a set of positive volume.
The necessity  of positive volume but not full volume 
is due to the volume preserving
KAM phenomenon  (see e.g. Cheng-Sun~\cite{Cheng-sun},
Herman~\cite{Herman} and Xia~\cite{Xia}),
which prevents the density of complete nonuniform hyperbolicity
near Diophantine integrable systems.
Although Pesin's conjecture is still widely open, 
partial answers were given in some peculiar situations, for instance,  
Liang and the third author~\cite{LY} 
proved the density of $\cH\cap \Diff^{r}(\cM, \vol_\cM)$ of $\Diff^{r}(\cM, \vol_\cM)$
under the $C^1$ topology.
Our result in Theorem~\ref{Thma} can be viewed as
a counterpart of Pesin's conjecture near the identity,
that is,  a non-ergodic completely
nonuniformly hyperbolic diffeomorphism arises 
in an arbitrarily small $C^r$ neighborhood of the identity,
exhibiting countably infinitely many ergodic components.

\subsection{Key Reduction and Strategy of Proof}

To construct a system with multiple ergodic components, 
it is more convenient to 
work on the $m$-dimensional cube $\cQ^m:=[-1, 1]^m$
rather than the Euclidean disk.
The key reduction in the proof of Theorem~\ref{Thma} 
is to make a construction for the particular case when
$\cM=\cQ^m$ and $\ell=1$, 
i.e., to construct 
a volume preserving nonuniformly hyperbolic Bernoulli diffeomorphism $f$ of $\cQ^m$
for $m\ge 5$. 
Besides the $C^r$-closeness to the identity,
$f$ is also required to be {\it flat} 
near the boundary $\p\cQ^m$ with respect to an a priori  given {\it admissible sequence}.
Such notion was introduced by Katok~\cite{K} for the Euclidean disk,
and we shall adapt  it to the  $m$-dimensional cube 
$\cQ^m$, with slight modifications for our purpose.
The precise definition is postponed to 
Definition~\ref{def flat}.

\begin{thmx}
\label{Thmb}
Let $m\ge 5$. 
For any $\ve>0$, any $r\in \NN$ and any admissible sequence $\rho$,
there exists a $C^\infty$ diffeomorphism $f$ of $\cQ^m$ 
with the following properties:
\begin{enumerate} 
\item  
$f$ is volume preserving, i.e.,
$f$ preserves the standard volume $\vol_{\cQ^m}$; 
\item  
$f$ is nonuniformly hyperbolic, i.e.,
$f$ has non-zero Lyapunov exponents at $\vol_{\cQ^m}$-almost every point $x\in \cQ^m$; 
\item $f$ is ergodic (in fact,  Bernoulli) with respect to $\vol_{\cQ^m}$;
\item 
\begin{enumerate}
\item[(i)] 
$\|f-\id\|_{C^r}\le \ve$, i.e.,
$f$ is $\ve$-close to 
the identity in the $C^r$ topology;
\item[(ii)]  $f\in \Diff^\infty_{\rho}(\cQ^m)$, i.e.,
$f$ is $\rho$-flat on $\cQ^m$ (see Definition~\ref{def flat}).
\end{enumerate}
\end{enumerate} 
\end{thmx}

We stress that the above property (4)  
is crucial for us to deduce Theorem~\ref{Thma} from Theorem~\ref{Thmb}.
Using such property, 
we first construct a $C^\infty$
volume preserving nonuniformly hyperbolic 
diffeomorphism of $\cQ^m$,
which has exactly $\ell$ ergodic components and is $\ve$-close to the identity in the $C^r$ topology;
then applying an embedding result of Katok in~\cite{K}
(see Proposition~\ref{prop: Katok top}), we 
obtain a $C^\infty$ diffeomorphism of arbitrary compact manifold $\cM$, which satisfies 
all the properties listed in Theorem~\ref{Thma} (see Section~\ref{sec: B to A 2}).

Note that Katok~\cite{K} and Dolgopyat-Pesin~\cite{DP} already proved 
that every compact manifold $\cM$ of   $\dim \cM\ge 2$
admits a diffeomorphism satisfying Properties (1)-(3) of Theorem~\ref{Thmb}.
However, Property (4) of 
Theorem~\ref{Thmb}, i.e.,
the $C^r$-closeness to the identity and the sufficient flatness near the boundary,
could not be obtained from the constructions in ~\cite{K, DP},
since the start-up diffeomorphism therein is away from the identity.

To this end, 
we recall the construction of a nonuniformly hyperbolic Bernoulli flow on every manifold
by Hu, Pesin and Talitskaya  in~\cite{HPT}, in which 
the essential step is to construct a volume preserving flow $\varphi^t$ 
on a special manifold $\cN$ (see Theorem~\ref{thm: HPT} in Section~\ref{sec: start-up}),
such that the following properties hold:
\begin{itemize}
\item all Lyapunov exponents of $\varphi^t$ are non-zero except the one along flow direction;
\item each time $t$-map of $\varphi^t$ has the {\it essential accessibility} property (see Section~\ref{sec: gentle}),
from which the ergodicity and Bernoulli property can be established;
\item $\varphi^t$ is generated by a vector field $X$ that is {\it sufficiently flat} near the boundary $\p \cN$.
\end{itemize}
Based on this preliminary construction, we shall prove Theorem~\ref{Thmb} 
along the following strategy:
we start with 
the time-$t$ map of the above flow, still denoted as $\varphi^t$,
then  we  make a {\it gentle} perturbation $h_{t\sigma}$ of  $\varphi^t$  in the case when $\dim \cN\ge 5$, such that
\begin{itemize}
\item[(i)] 
$h_{t\sigma}$ is close to the identity 
in the $C^r$ topology  for sufficiently small $t$ and $\sigma$;
\item[(ii)] the essential accessibility property and the flatness near $\p \cN$ are maintained for $h_{t\sigma}$;
\item[(iii)] 
$h_{t\sigma}$ has a positive average Lyapunov exponent in the central direction.
\end{itemize} 
Finally, we conjugate $h_{t\sigma}: \cN\to \cN$ to our target map
$f: \cQ^m\to \cQ^m$, which satisfies all the  properties in Theorem~\ref{Thmb}.

\section*{{\bf Acknowledgement}}
We would like to express our gratitude to A.Wilkinson and  D.Dolgopyat for helpful conversations in the preparation of this paper. 
J.Chen is partially supported by the National Key Research and Development Program of China (No. 2022YFA1005802),
the NSFC Grant 12471186, NSF of Jiangsu BK20241916 
and the Jiangsu Shuang Chuang Project (JSSCTD 202209).
Y.Yang is partially supported by the National Science Foundation under Award No. DMS-2000167.


\section{Theorem~\ref{Thmb} Implies Theorem~\ref{Thma}}
\label{sec: B to A}

\subsection{Preliminaries from Differential Topology}
\label{sec: B to A 1}

The notion of flatness is stated by Katok~\cite{K} for  the Euclidean disk, 
yet it is easy to be adapted on the $m$-dimensional cube 
$\ds
\cQ^m=[-1, 1]^m=
\left\{x\in \RR^m:  \|x\|_\infty \le 1 \right\},
$
where $\|x\|_\infty:=\max\{|x_1|,  \dots,  |x_m|\}$ for any $x=(x_1, \dots, x_m)\in \RR^m$.
That is, we make slight modifications on the original definitions  in Section 1.4 of~\cite{K}:
we require the admissible sequence vanishes on $\p\cQ^m$,
replace the Euclidean norm by $\infty$-norm,
and fix $\ve_n$ to be $2^{-n}$ to
control the distance away from $\p\cQ^m$.

\begin{definition}[Admissible Sequence and Flatness]
\label{def flat}
A sequence $\rho=(\rho_0, \rho_1, \dots)$ of real-valued continuous functions on $\cQ^m$
is said to be
admissible if every function  $\rho_n$ is 
strictly positive on the interior $\Int(\cQ^m)$ and vanishes on the boundary $\p \cQ^m$.
The class of $\rho$-flat  functions on $\cQ^m$ is defined  by
\beqn
\begin{split}
 C^\infty_{\rho}(\cQ^m):=
\bigg\{ \phi\in C^\infty(\cQ^m):   \  & 
\left| \frac{\p^n \phi(x_1, \dots, x_m)}{\p^{i_1} x_1 \dots \p^{i_m} x_m}\right|\le \rho_n(x_1, \dots, x_m)
\ \text{holds for} \  
\\
& \hspace{-1.5cm} \forall n\ge 0, \ \forall x= (x_1, \dots, x_m)\in \cQ^m \ \text{with} \ \|x\|_\infty \ge 1-2^{-n}, \\
&  \forall  i_1, \dots, i_m\in \NN\cup\{0\} \ \text{with} \  i_1+\dots + i_m=n
\bigg\}.
\end{split}
\eeqn 
Further,  the class of $\rho$-flat 
diffeomorphisms on $\cQ^m$
is defined by
\beqn
\Diff^\infty_{\rho}(\cQ^m):=
\left\{ f\in \Diff^\infty(\cQ^m): \ f_i(x_1,\dots, x_m)-x_i\in C^\infty_{\rho}(\cQ^m), \ \forall i=1, \dots, m
\right\};
\eeqn
and the class of $\rho$-flat 
vector fields on $\cQ^m$
is defined by
\beqn
\Gamma^\infty_{\rho}(\cQ^m):=
\left\{ V\in \Gamma^\infty(\cQ^m): \ V_i(x_1,\dots, x_m)\in C^\infty_{\rho}(\cQ^m), \ \forall i=1, \dots, m
\right\}.
\eeqn
\end{definition}

It follows that any function $\phi\in C^\infty_{\rho}(\cQ^m)$ (or any vector field $V\in \Gamma^\infty_{\rho}(\cQ^m)$)
and all its partial derivatives of any order vanish on the boundary $\p\cQ^m$;
and any diffeomorphism
$f\in \Diff^\infty_\rho(\cQ^m)$
is $C^\infty$ tangent to the identity near    $\p \cQ^m$, 
i.e., $d^n(f-\id)|_{\p\cQ^m}=0$ for any $n\in \NN\cup\{0\}$.
In particular, $f|_{\p \cQ^m}\equiv \id$.

The following proposition was originally proved by Katok in~\cite{K} for  the Euclidean disk, 
and we can easily modify it for the $m$-dimensional cube $\cQ^m$.

\begin{prop}[Katok~\cite{K}, Proposition 1.1 and 1.2]
\label{prop: Katok top}
Let $\cM$ be a compact connected smooth Riemannian manifold (possibly with boundary)
of $\dim \cM=m\ge 2$.
\begin{enumerate}
\item
There exists a continuous mapping $\Psi=\Psi_{\cQ^m\to \cM}: \cQ^m\to \cM$ such that
\begin{enumerate}
\item the restriction $\ds \Psi|_{\Int(\cQ^m)}$ is a $C^\infty$ diffeomorphic embedding;
\item $\ds \Psi(\cQ^m)=\cM$ and $\ds \vol_\cM\left(\cM\backslash \Psi\left( \Int(\cQ^m)\right)\right)=0$;
\item $\ds \Psi$ is volume preserving, i.e., $\Psi_* \vol_{\cQ^m}=\vol_\cM$.
\end{enumerate}
\item
Furthermore, there exists an admissible sequence $\rho=\rho_\Psi$ on $\cQ^m$
such that 
if  $f\in \Diff^\infty_{\rho}(\cQ^m)$ is volume preserving,
then the map $F=F_{f, \Psi}: \cM\to \cM$ defined by
\beq
\label{def Ff}
F(x)=F_{f, \Psi}(x):=
\begin{cases}
\Psi\circ f\circ \Psi^{-1}(x), & \text{if}\  x\in \Psi\left( \Int(\cQ^m)\right) \\
x, & \text{otherwise}
\end{cases}
\eeq
is a $C^\infty$ volume preserving diffeomorphism of $\cM$.
\end{enumerate}
\end{prop}

Let us make some comments on Statement (2) in the above.
The derivatives of the almost conjugacy
$\Psi$  may blow up very fast near $\p\cQ^m$,
and thus the admissible sequence  $\rho=\rho_\Psi$ should
vanish relatively faster near $\p\cQ^m$ to ensure 
 the $C^\infty$ smoothness of $F$. 
We further notice that 
if an admissible sequence $\rho=(\rho_0, \rho_1, \dots)$ 
satisfies Statement (2)  of Proposition~\ref{prop: Katok top},
so does $\wrho=(\wrho_0, \wrho_1, \dots)$
with $\wrho_n\le \rho_n$ for every $n\ge 0$.
Hence for any $\ve>0$ and any $r\in \NN$,
we may choose an admissible sequence $\rho=\rho_{\Psi, \ve, r}$
such that for any  $f\in \Diff^\infty_{\rho}(\cQ^m)$,
the corresponding diffeomorphism  $F_{f, \Psi}$ given by \eqref{def Ff}
satisfies
\beqn
\dist_{C^r}\left( F_{f, \Psi}|_{\Psi(\cQ_r^m)}, \id \right)<\ve,
\ \ \text{where} \ \
\cQ_r^m:=\left\{x\in \RR^m: 1-2^{-r}\le \|x\|_\infty\le 1\right\}.
\eeqn
It then follows that there exists a constant $C_{\Psi, r}>0$ such that 
\beq
\label{eq: Cr est F id}
\begin{split}
\dist_{C^r}\left( F_{f, \Psi}, \id \right)
& = \max\left\{ \dist_{C^r}\left( F_{f, \Psi}|_{\Psi(\cQ_r^m)}, \id \right),
\dist_{C^r}\left( F_{f, \Psi}|_{\Psi(\cQ^m\backslash\cQ_r^m)}, \id \right)
\right\} \\
& \le \max\left\{\ve, C_{\Psi, r}\|f-\id\|_{C^r}.
\right\}
\end{split}
\eeq

\subsection{Proof of Theorem~\ref{Thma}}
\label{sec: B to A 2}

Suppose now  that Theorem~\ref{Thmb} holds, 
and we show how
Theorem~\ref{Thma} can be deduced  from Theorem~\ref{Thmb}.

Let $\cM$ be a  compact connected smooth Riemannian manifold (possibly with boundary)
of dimension  $m\ge 5$,
and let $\Psi=\Psi_{\cQ^m\to \cM}$ be the almost conjugacy 
obtained by Statement (1) of Proposition~\ref{prop: Katok top}.
For any $\ve>0$ and any $r\in \NN$,
let $\rho=\rho_{\Psi, \ve, r}=(\rho_0, \rho_1, \dots)$ be the admissible sequence
and
let $C_{\Psi, r}$ be the constant 
such that~\eqref{eq: Cr est F id} holds.

Given $\ell\in \NN\cup \{\infty\}$, 
we set $a_\ell=1$ and
$\ds
a_k=1-2^{-k+1}
$
for $1\le k< \ell$.
One can then slice the   cube
$\cQ^m=[-1, 1]^m$ into 
$\ell$ rectangular boxes $\{\cQ_k^m\}_{1\le k\le \ell}$, where
$\ds
\cQ_k^m = [-1, 1]^{m-1} \times [a_{k-1}, a_k]
$
for $1\le k\le \ell$
(note that
$\ds
\cQ^m_\infty=[-1, 1]^{m-1} \times \{1\}
$
is degenerate).
Each $\cQ_k^m$ is an affine image of $\cQ^m$
by a linear scaling of factor $2^{-k}$ in the last coordinate, 
that is,
$\cQ^k_m=\pi_k(\cQ^m)$, where  
$\ds
\pi_k(x_1, \dots, x_{m-1}, x_m)=(x_1, \dots, x_{m-1}, a_{k-1}+ 2^{-k} (x_m + 1) ).
$

Now for every $k\in \NN\cap [1, \ell]$, applying Theorem~\ref{Thmb}
with respect to  
\beq
\label{def vek rhok}
\ve_k= \ve  C_{\Psi, r}^{-1}\cdot 4^{-k^2r}, \ \
r_k=kr, \ \  \text{and} \ \ 
\rho^k=(\rho^k_0, \rho^k_1, \dots)
\ \text{with} \ \rho^k_n=\rho_n \cdot 2^{-kn},
\eeq
we obtain a
$C^\infty$ volume preserving nonuniformly hyperbolic Bernoulli
diffeomorphism $f_k$ of $\cQ^m$ such that
$\ds \|f_k-\id\|_{C^{kr}}\le \ve_k$ and
$\ds f_k\in \Diff^\infty_{\rho^k}(\cQ^m)$.
Then we define a map $f: \cQ^m\to \cQ^m$ be setting
\beqn
\label{def f}
f(x)=
\begin{cases}
\pi_k\circ f_k\circ \pi_k^{-1}(x), \ & \text{if} \ x\in \cQ^m_k, \\
x, \ &  \text{otherwise}.
\end{cases}
\eeqn
We claim that
$f$ is a $C^\infty$ diffeomorphism of $\cQ^m$. 
This is obvious when $\ell<\infty$;
in the case when $\ell=\infty$, 
our construction implies that
$f$ is $C^\infty$ on $[-1, 1]^{m-1}\times [-1, 1)$,
and it is $C^\infty$ at the side $[-1, 1]^{m-1}\times \{1\}$
since
\beqn
\|f|_{\cQ^m_k}-\id\|_{C^{kr}}
=\|\pi_k \circ (f_k-\id)\circ \pi_k^{-1} \|_{C^{kr}}
\le (2^k)^{kr} \|f_k-\id\|_{C^{kr}} \le 
\ve C_{\Psi, r}^{-1}\cdot 2^{-k^2 r}
\eeqn
vanishes as $k\to\infty$. The above estimate also provides that 
\beq
\label{dist f id}
\|f-\id\|_{C^{r}}\le \sup_{1\le k\le \ell}
\|f|_{\cQ^m_k}-\id\|_{C^{kr}}
\le  \ve C_{\Psi, r}^{-1}.
\eeq
Moreover, 
since $\ds f_k\in \Diff^\infty_{\rho^k}(\cQ^m)$, by
the definition of the admissible sequence $\rho^k$ in~\eqref{def vek rhok},
we have $f\in \Diff^\infty_{\rho}(\cQ^m)$
for $\rho=\rho_{\Psi, \ve, r}$.

Finally, we apply  Proposition~\ref{prop: Katok top}
and obtain a $C^\infty$ volume preserving diffeomorphism $F=F_{f, \Psi}: \cM\to \cM$.
It is also clear from our construction
that $F$ is nonuniformly hyperbolic,
and for each $ k\in \NN\cap [1, \ell]$, the set
$\Psi(\cQ^m_k)$ is an ergodic (in fact, Bernoulli) component of $F$.
Moreover, it follows from~\eqref{eq: Cr est F id} and 
~\eqref{dist f id} that $\dist_{C^r}\left( F, \id \right)\le \ve$.
The proof of Theorem~\ref{Thma} is now complete.


\section{Construction and Proof of Theorem~\ref{Thmb}}
 
\subsection{The Start-up Diffeomorphism}
\label{sec: start-up}
 
To prove Theorem~\ref{Thmb}, we start with the time-$t$ map $\varphi^t$ of 
a volume preserving
nonuniformly hyperbolic Bernoulli flow on a special manifold $\cN$, 
which was constructed by Hu, Pesin and Talitskaya in~\cite{HPT}.

To this end,
we first recall the Katok map
$g$ of the two-dimensional Euclidean disk
$\ds
\cD^2:=\{(x_1, x_2)\in \RR^2:  x_1^2+x_2^2\le 1\},
$
which was constructed in~\cite{K} as a 
prototype of nonuniformly hyperbolic surface diffeomorphisms.
Similar to Definition~\ref{def flat},
the notion of flatness for functions and maps on $\cD^2$ was introduced in Section 1.4 of~\cite{K}, 
and it can also be defined for smooth vector fields on $\cD^2$ in a similar fashion.

\begin{thm}[Katok~\cite{K}, Theorem A]
\label{thm: Katok map}
For any admissible sequence $\rho$ on $\cD^2$, 
there exists an area preserving nonuniformly hyperbolic Bernoulli diffeomorphism 
$g\in \Diff^\infty_\rho(\cD^2)$.
\end{thm}

We denote by $\cS_g$
the singularity set of the Katok map $g$, 
which consists of $\p\cD^2$ and three fixed points of $g$. 
It was shown by Dolgopyat and Pesin in Proposition 2.2 of~\cite{DP} that 
there exist continuous invariant stable/unstable cone families, distributions and foliations 
for $g$ in $\cD^2\backslash \cS_g$,
with uniform control on any compact subset of $\cD^2\backslash \cS_g$.

Applying a well known result of Smale (see~\cite{S}, Theorem B), i.e.,
the space of $C^\infty$ diffeomorphisms of $\cD^2$
which are the identity in some neighborhood of $\p\cD^2$
is contractible,
Hu, Pesin and Talitskaya proved in \cite{HPT} that there exists 
a smooth isotopy connecting  the identity map and the Katok map,
with some additional properties.

\begin{prop}[Hu-Pesin-Talitskaya~\cite{HPT}, Proposition 3 and 4]
\label{prop: Katok isotopy}
For any admissible sequence $\rho$ on $\cD^2$, 
let $g\in \Diff^\infty_\rho(\cD^2)$ be the Katok map given in Theorem~\ref{thm: Katok map}.
There exists a $C^\infty$ map $G: \cD^2 \times [0,1] \to \cD^2$ such that
\begin{enumerate}
\item 
$G(\cdot, 0)=\id$ and $G(\cdot, 1)=g$. Moroever, 
$d^kG(x,1)=d^kG(g(x),0)$ for any $k\ge 0$;
\item for each $t\in [0,1]$, the map $G(\cdot, t): \cD^2 \to \cD^2$ is an
area preserving diffeomorphism;
\item
there exists a neighborhood $\cU$ of $\p\cD^2$ and 
a $\rho$-flat  vector field $V$ in $\cD^2$, 
such that 
$G(\cdot, t)|_{\cU}$ is the time-$t$ map of the flow generated by $V$ restricted to $\cU$.
\end{enumerate}
\end{prop}

\begin{remark}
\label{rem: dim2}
Statement (3) of Proposition~\ref{prop: Katok isotopy} implies that $g_t:=G(\cdot, t)\in \Diff^\infty_\rho(\cD^2)$ for any $t\in [0, 1]$,
and $g_t$ can be arbitrarily close to the identity for sufficiently small $t$. 
However, Proposition~\ref{prop: Katok isotopy} does not guarantee
the hyperbolicity and ergodicity of the map $g_t: \cD^2\to \cD^2$ for $t\in (0, 1)$,
which impedes us to prove Theorem~\ref{Thma} and~\ref{Thmb}  in dimension $2$.
\end{remark}

We also need Brin's construction from~\cite{B}.
Given $m\ge 5$, let $A: \TT^{m-3}\to \TT^{m-3}$ be the 
hyperbolic automorphism  
induced by a block diagonal matrix 
$\diag\{A_1, \cdots, A_{m'}\}$,
where 
$m'=[(m-3)/2]$,
$\ds A_i=
\begin{pmatrix}
2 & 1 \\ 1 & 1
\end{pmatrix}
$
for $1\le i<m'$,
while
$\ds 
A_{m'}=
\begin{pmatrix}
2 & 1 \\ 1 & 1
\end{pmatrix}
$ if $m$ is odd
and
$\ds 
A_{m'}=
\begin{pmatrix}
2 & 1 & 1 \\ 1 & 1 & 1 \\ 0 & 1 & 2
\end{pmatrix}
$
 if $m$ is even.
The mapping torus of   $A$  
is defined by 
\beq
\label{def cL}
\cL:=\TT^{m-3}\times [0, 1]/\sim \,
=\left\{ (y, \tau): \ y\in \TT^{m-3}, \ \tau\in [0, 1]   \right\}/
\left\{ (y, 1)\sim (Ay, 0)\right\}.
\eeq
Now we recall the special manifold $\cN$ introduced in~\cite{HPT}
for the case $\dim \cN=m\ge 5$, i.e.,
\beq
\label{def cN}
\cN:=\cD^2\times \cL
=\left\{ (x, y, \tau): \ x \in \cD^2, \ (y, \tau) \in \cL \right\}.
\eeq
Let $\cK$ be the mapping torus of $g\times A$, where $g$ is the Katok map
and 
$A$ is the hyperbolic toral automorphism  from Brin's construction,
i.e., $\cK:=\cD^2\times \TT^{m-3}\times [0, 1]/\sim$, where $\sim$ is the identification $(x, y, 1)=(g(x), Ay, 0)$.
Then $\cK$ is diffeomorphic to $\cN$ via the diffeomorphism $\cG: \cK\to \cN$ given by
$\ds
\cG(x, y, \tau)=(G(x, \tau), y, \tau)
$,
where 
$G$ is the smooth isotopy given by Proposition~\ref{prop: Katok isotopy}. 

We further recall the $C^\infty$ smooth vector field $X$ on $\cN$ introduced in~\cite{HPT}, that is,
\beq
\label{def X}
X(G(x, \tau), y, \tau):=\left( \frac{\p G}{\p \tau}(x, \tau), 0, \alpha(G(x, \tau))\right), \ \text{for any} \ (x, y, \tau)\in \cK,
\eeq
where the function $\alpha: \cD^2\to [0, 1]$ is chosen with the following properties:
\begin{enumerate}
\item[(A1)] $\alpha$ and all its partial derivatives of any order vanish on $\p\cD^2$;
\item[(A2)] $\alpha(x)>0$ for any $x\in \Int(\cD^2)$, and $\alpha(x)=1$ for any $x\in \cD^2\backslash \cU$;
\item[(A3)] $\alpha(x)^{-1} V(x)\to 0$ as $x\to \p\cD^2$,
\end{enumerate}
where 
the neighborhood $\cU$ of $\p\cD^2$ 
and the vector field $V$ on $\cD^2$ are  
given by Proposition~\ref{prop: Katok isotopy}.
For our purpose, we assume a stronger property than (A1), that is,
\begin{enumerate}
\item[(A1$'$)] 
$\alpha\in C^\infty_{\brho}(\cD^2)$, where 
$\brho=(\brho_0, \brho_1, \dots)$ is any a priori 
given  admissible sequence on $\cD^2$ with range in $[0, 1]$,
and $\rho=\brho^2:=(\brho_0^2, \brho_1^2, \dots)$ is the admissible sequence chosen
in the assumption of Theorem~\ref{thm: Katok map} and Proposition~\ref{prop: Katok isotopy}.
\end{enumerate}

\begin{thm}[Hu-Pesin-Talitskaya~\cite{HPT}, Lemma 7 and 9]
\label{thm: HPT}
Let $\cN$ be the manifold given by \eqref{def cN} with $\dim\cN=m\ge 5$,
and
let $\varphi^t$ be the time-$t$ map of the flow on $\cN$ 
generated by the vector field $X$  in \eqref{def X},
which satisfies Conditions (A1$'$), (A2) and (A3).
Then the following statements hold for any $t\in (0, 1]$:
\begin{enumerate}
\item $\varphi^t$ preserves the volume $\vol_{\cN}$ 
(since $X$ is divergence free);
\item 
all Lyapunov exponents of $\varphi^t$ are non-zero, except the one along the flow direction,
for $\vol_{\cN}$-almost every point;
\item $\varphi^t$ is ergodic (in fact,  Bernoulli) with respect to $\vol_{\cN}$;
\item $\varphi^t \in \Diff^\infty_{\brho}(\cN)$, i.e., $\varphi^t$ is $\brho$-flat
(since the vector field $X$ is chosen to be $\brho$-flat).
\end{enumerate}
\end{thm}

We remark that Statement (4) of Theorem~\ref{thm: HPT} is not proposed in~\cite{HPT},
and it is in fact due to Condition (A1$'$). 
Indeed, 
the flatness  only concerns the behavior near the boundary $\p\cN=\p \cD^2\times \cL$,
and Statement (3) of Proposition~\ref{prop: Katok isotopy} implies that
$\ds 
X(x, y, \tau)=(V(x), 0, \alpha(x))
$
for any $(x, y, \tau)\in \cU\times \cL$,
from which the $\brho$-flatness of $X$ follows.

\subsection{Further Properties of $\varphi^t$ and Its Gentle Perturbations}
\label{sec: gentle}

The notion of {\it pointwise partially hyperbolic diffeomorphisms}
was  introduced by Hu, Pesin and Talitskaya
in~\cite{HPT13}  (see also \cite{CHP13} and \cite{CHY20}).
From the construction of the special manifold $\cN$ in \eqref{def cN}
and  the vector field $X$ in \eqref{def X},
it is easy to verify that the time-$t$ map  $\varphi^t$  is
a pointwise partially hyperbolic diffeomorphism  on the open subset 
$\cN_0:=(\cD^2\backslash \cS_g)\times \cL\subset \cN$,
where $\cS_g$ is the singularity set of the Katok map $g$.
That is, there exists a $d\varphi^t$-invariant splitting 
\beq
\label{split PPH}
T_z\cN = E_X^s(z) \oplus E_X^c(z) \oplus E_X^u(z),
\ \ \text{for every} \ z\in \cN_0,
\eeq
and there are continuous functions
$\lambda(z)<\lambda'(z)\le 1\le \mu'(z)<\mu(z)$ defined for $z\in \cN_0$,
such that 
\beqn
\begin{split}
\|d\varphi^t(z) v\|\le \lambda(z)^t \|v\|,  \ \ & \ \  v\in E^s_X(z), \\
\lambda'(z)^t \|v\| \le 
\|d\varphi^t(z) v\|\le \mu'(z)^t \|v\|,  \ \ &  \ \ v\in E^c_X(z), \\
\mu(z)^t \|v\| \le \|d\varphi^t(z) v\| \hspace{1.9cm},  \ \ &  \ \ v\in E^u_X(z).
\end{split}
\eeqn
Here
$\dim E_X^s(z) =m'+1$, $\dim E_X^c(z) =1$,  and 
$\dim E_X^u(z) =m-m'-2$, where $m'=[(m-3)/2]$. 
In fact, 
$d\varphi^t$ acts isometrically on the flow central bundle $E_X^c$, 
i.e., $\lambda'(\cdot)=\mu'(\cdot)\equiv 1$.

Using the similar arguments as in Section 3 of \cite{DP}, 
one can actually show that 
there exist continuous invariant stable/unstable cone families for $\varphi^t$ 
in $\cN_0$, 
with uniform control on any compact subset of $\cN_0$. 
Furthermore, the following properties hold:

\begin{propertyx}
\label{P1}
$\varphi^t$  has strongly stable and unstable $(\delta, q)$-foliations
$W^s_X$ and $W^u_X$, where $\delta$ and $q$ 
are continuous functions on $\cN_0$ (see the precise definition in Section 2 of \cite{HPT});
\end{propertyx}

\begin{propertyx}
\label{P2}
the foliations $W^s_X$ sand $W^u_X$ are absolutely continuous;
\end{propertyx}

\begin{propertyx}
\label{P3}
$\varphi^t$ has negative Lyapunov exponents in the direction of $E^s_X$
and positive Lyapunov exponents in the direction of $E^u_X$ almost everywhere.
\end{propertyx}

In the proof of Statement (3) of Theorem~\ref{thm: HPT}, 
a key ingredient is the {\it essential accessibility} property of the time-$t$ map
$\varphi^t: \cN\to \cN$.
In fact, it was shown in Lemma 8 of~\cite{HPT}
that

\begin{propertyx}
\label{P4}
$\varphi^t: \cN_0\to \cN_0$ 
has the {\it accessibility} property via the foliations $W^s_X$  and $W^u_X$,
that is, any two points $z, z'\in \cN_0$ are {\it accessible}, i.e., 
there are points $z=z_0, z_1, \dots, z_{k-1}, z_k=z'$ in $\cN$,
such that $z_i\in W^*_X(z_{i-1})$ for $i=1, \dots, k$ and $*=s$ or $u$.
\end{propertyx}

In the next subsection, we shall construct a special {\it gentle perturbation} of $\varphi^{t}$.
Here `gentle' means that the perturbation only occurs in a 
domain $\Delta$ which is strictly inside $\cN_0$.
That is, we say $h: \cN\to \cN$ is a gentle perturbation of $\varphi^t$ on $\Delta$, if 
$h(z)= \varphi^t(z)$ for all $z\in \cN\backslash \Delta$. 
For our purpose, we also require that 
$\Delta$ is away from $\cU\times \cL$,
where $\cU$ is the neighborhood $\p\cD^2$ 
given by Proposition~\ref{prop: Katok isotopy}.

\begin{prop}
\label{prop: gentle 0}
For any $t\in (0, 1]$ and any
compact domain $\Delta$  inside 
$\Int\left(\cN_0\backslash (\cU\times \cL)\right)$,
there exists $\delta_{t, \Delta}>0$ such that 
if $h$ is  a gentle perturbation of $\varphi^t$ on $\Delta$,
which satisfies 
$\ds \|h-\varphi^t\|_{C^1} <\delta_{t, \Delta}$,
then
$h$ is a pointwise partially hyperbolic diffeomorphism on $\cN_0$ such that
Properties~\ref{P1}-~\ref{P4} hold for $h$.
\end{prop}

\begin{proof}
Utilizing the uniformity of invariant cone families, distributions and foliations
of $\varphi^t$ on any compact domain inside $\cN_0$,  
it is routine to show the pointwise partial hyperbolicity and 
Properties~\ref{P1}-~\ref{P3} for any gentle perturbation $h$
of $\varphi^t$ on $\Delta$, as long as $h$ is sufficiently close to $\varphi^t$
in the $C^1$ topology. 
We refer to Section 5 of~\cite{DP} for similar arguments,
and  leave the details to the readers.

It remains to show~\ref{P4} for $h$. To this end,
we recall the arguments in the proof of the accessibility 
for $\varphi^t: \cN_0\to\cN_0$ (see the proof of Lemma 8 in~\cite{HPT}, 
in which a different flow that is equivalent to $\varphi^t$ was discussed):
via the stable and unstable foliations of $\varphi^t$,
\begin{itemize}
\item[(AC1)] 
every point $z=(x, y, \tau)\in \cN_0$ is accessible to a point in 
$\ds 
\Pi_{p, q}:=\left\{(p, q, \tau'):  \tau'\in [0, 1]\right\},
$
where $p\in \cD^2$ is a periodic point of the Katok map $g$ in $\cD^2\backslash \cU$,
and $q\in \TT^{m-3}$ is a periodic point of the hyperbolic 
automorphism $A$ in Brin's construction;
\item[(AC2)] 
let $p'$ be another periodic point of the Katok map $g$ in $\cU$ such that $\alpha(p')<1$,
where $\alpha$ is the function on $\cD^2$ given below \eqref{def X}.
It follows that for any $z=(p, q, \tau)\in \Pi_{p, q}$, 
there is a four-legged $us$-path 
consisting of four curves moving inside  
$W^u_X(p, q, \tau)$,
$W^s_X(p', q, \tau_1)$,
$W^u_X(p', q, \tau_1)$,
$W^s_X(p, q, \tau_2)$
successively, 
such that $\tau_2<\tau_1<\tau$.
That is, 
any point $z=(p, q, \tau)$ is accessible to another point 
$z_2=(p, q, \tau_2)$ with $\tau_2<\tau$.
\item[(AC3)]
by shrinking the leg  length of the above $us$-path continuously, 
we conclude that any point $z=(p, q, \tau)\in \Pi_{p, q}$ 
is accessible to all the points 
$(p, q, \tau')$ with $\tau'\in [\tau_2, \tau]$.
It follows that any two points in $\Pi_{p, q}$ are accessible.
Together with (AC1), every point $z\in \cN_0$ is accessible to the point $(p, q, 0)$.
\end{itemize}
Now let $h$ be  a gentle perturbation of $\varphi^t$ on $\Delta$,
where $\Delta$ is a compact domain strictly inside 
$\Int\left(\cN_0\backslash (\cU\times \cL)\right)$.
Due to the abundance of periodic orbits for the Katok map $g$,
we may assume the periodic points $p$ and $p'$ are chosen away from $\Delta$,
and so are $\Pi_{p, q}$ and the four-legged $us$-path. 
By continuity of the stable/unstable foliations under gentle perturbations,
(AC1)(AC2)(AC3) would still hold for $h$,
if $h$ is sufficiently close to $\varphi^t$ in the $C^1$ topology.
That is, any $z\in \cN_0$ is accessible to $(p, q, 0)$
via the stable and unstable foliations of $h$,
and hence $h:\cN_0\to \cN_0$ has accessibility property,
i.e. Property~\ref{P4} holds for $h$.
\end{proof}

\subsection{Special Gentle Perturbation with Positive Central Lyapunov Exponent}
\label{sec: construct h}
 
Formula~\eqref{def X} indicates that the vector field $X$  
is independent of the $y$-components, and thus $d\varphi^t$
preserves any $A$-invariant vector bundle,
where $A$ is the hyperbolic  automorphism from Brin's construction. 
In particular, $d\varphi^t$ preserves the linear stable bundle $E^{s, m'}_{X}$
and linear unstable bundle $E^{u, m'}_X$
corresponding to the $m'$-th block matrix $A_{m'}$,
where $m'=[(m-3)/2]$.
Note that 
$\dim E^{s, m'}_{X}=1$,
and 
$\ds \dim E^{u, m'}_X
=\begin{cases}
1, &\text{if} \ m \ \text{is odd}, \\
2, &\text{if} \ m \ \text{is even}.
\end{cases}
$
In the latter case, 
there is a further $d\varphi^t$-invariant splitting
$E^{u, m'}_X=E^{u, m',1}_X\oplus E^{u, m',2}_{X}$,
as $A_{m'}$ has two distinct eigenvalues of modulus greater than $1$. 
Recall that $E^c_X$ is the central  bundle generated by the vector field $X$, which is 
a one-dimensional $C^\infty$ smooth $d\varphi^t$-invariant bundle.
We shall particularly work on the following 
two-dimensional $C^\infty$ smooth $\varphi^t$-invariant bundle:
\beq
\label{def cE}
\cE=\cE^u\oplus \cE^c,
\ \ \text{where} \ \
\cE^u :=
\begin{cases}
E^{u, m'}_X, & \text{if} \ m \ \text{is odd}, \\
E^{u, m', 1}_X, & \text{if} \ m \ \text{is even},
\end{cases}
\ \text{and} \ 
\cE^c:=E^c_X.
\eeq

The above observation
allows us to assign a smooth local coordinate system $\ds (\xi^u, \xi^c, \zeta)$
on a neighborhood centered at a point $z_0\in \cN_0$, 
where
\beq
\label{def Cart coord}
\langle\p/\p \xi^u\rangle=\cE^u, \
\langle\p/\p \xi^c\rangle=\cE^c, \
\zeta=
\begin{cases}
(x_1, x_2, y_1, \dots, y_{2(m'-1)}, \xi^s),  & \text{if} \ m \ \text{is odd}, \\
(x_1, x_2, y_1, \dots, y_{2(m'-1)}, \xi^s, \xi^{u, 2}),  & \text{if} \ m \ \text{is even}, 
\end{cases}
\eeq
in which 
$\ds (x_1, x_2)\in \cD^2, (y_1, \dots, y_{2(m'-1)})\in  \TT^{2(m'-1)}$,
$\ds \langle\p/\p \xi^s\rangle=E^{s, m'}_X$,
and 
$\ds
\langle\p/\p \xi^{u, 2}\rangle=E^{u, m', 2}_X
$
if $m$ is even.
We also need the cylindrical coordinate system 
$\ds
(\varrho, \theta, \zeta),
$
such that
$\ds
\xi^u=\varrho \cos\theta
$
and 
$\ds
\xi^c=\varrho \sin\theta.
$
For a sufficiently small $\gamma_0>0$,  we set
\beq
\label{def Delta}
\Delta=\Delta(z_0, \gamma_0)
:=\left\{ (\varrho, \theta, \zeta):  \
\varrho\in [0, \gamma_0],  \
\theta\in [0, 2\pi), \
\|\zeta\|\le \gamma_0\right\}, 
\eeq
where $\|\zeta\|$ denotes the Euclidean norm of $\zeta$. 

Given any $t\in (0, 1]$,
we introduce a special gentle perturbation of $\varphi^{t}$ as follows.
Due to the abundance of non-periodic orbits for $\varphi^t$,
we can choose $z_{t0}\in \cN_0$ and $\gamma_{t0}\in (0, 0.1)$
such that the   neighborhood $\Delta_t=\Delta(z_{t0}, \gamma_{t0})$ of the form~\eqref{def Delta}
satisfies the following property:
\beq
\label{def Delta t}
\begin{split}
& \varphi^{tj}(\Delta_t) \cap \Delta_t=\emptyset,   \ \text{for any} \ j\in [-N_t, N_t]\backslash\{0\}; \\
& \varphi^{tj}(\Delta_t) \cap \left(\left(\cS_g\cup \cU\right)\times \cL \right)=\emptyset,  \ \text{for any} \ j\in [-N_t, N_t],
\end{split}
\eeq
where $\cS_g$ is the singularity set of the Katok map,
$\cU$ is given by Proposition~\ref{prop: Katok isotopy},
and $N_t$ is a positive integer such that 
\beq
\label{choose Nt}
\eta^{tN_t}>100,
\eeq  
where
$\eta$ is the expansion rate of $A_{m'}$ along $\cE^u$.
Since the $\tau$-component
of the vector field $X$ given in \eqref{def X} 
is constantly one on $\ds \left( \cD^2\backslash \cU\right)\times \cL$,
we have 
\beq
\label{def expansion rate}
\|d\varphi^t(z)|\cE^u(z)\|=\eta^t,
\ \text{for any} \ z\in \varphi^{tj}(\Delta_t), \ \text{with} \ j\in [-N_t, N_t].
\eeq
We further choose two $C^\infty$ smooth functions
$\psi, \psi_1: \RR\to \RR$ such that
\begin{itemize}
\item  $\psi(w)>0$ for $0.1\gamma_{t0}< |w| <\gamma_{t0}$, and
$\psi(w)=0$ otherwise.
\item $\psi_1(w)>0$ for $|w|< \gamma_{t0}$, 
and $\psi_1(w)=0$ otherwise. Also, 
$\psi_1(w)$ is constant for 
$|w|< 0.5\gamma_{t0}$.
\end{itemize}
We then define a map $\phi_\sigma: \cN\to \cN$, with $\sigma\in [0, 1]$, by setting 
$\phi_\sigma=\id$ on $\cN\backslash \Delta_t$, 
and 
\beq
\label{def phi si}
\phi_\sigma(\varrho, \theta, \zeta) = (\varrho, \theta+\sigma 
\wpsi(\varrho, \zeta), \zeta ), 
\ \text{for any} \ (\varrho, \theta, \zeta)\in \Delta_t,
\eeq
where $\wpsi(\varrho, \zeta):=\psi(\varrho)\psi_1(\|\zeta\|)$.
Finally, we define
\beq
\label{def h t si}
h_{t\sigma}= \phi_\sigma \circ \varphi^t : \cN\to \cN.
\eeq
It is obvious that $\phi_\si$ and thus $h_{t\sigma}$ 
are volume preserving, i.e., they both preserve $\vol_\cN$.

\begin{prop}
\label{prop: property of h}
For any $\delta>0$ and any $r\in \NN$, 
there exists $(t, \sigma)\in (0, 1]^2$ such that 
the diffeomorphism 
$h_{t\si}:\cN\to \cN$ 
given in \eqref{def h t si}
satisfies the following properties:
\begin{enumerate}
\item 
$\ds
\|h_{t\sigma}-\id\|_{C^r}<\delta;
$
\item 
$h_{t\sigma}$ is a pointwise partially hyperbolic  
diffeomorphism on $\cN_0$; 
\item 
Properties~\ref{P1}-~\ref{P4} hold for $h_{t\sigma}$;
\item 
$h_{t\sigma}$ has positive average central Lyapunov exponent, that is,
\beq
\label{def neg central LE}
\int_\cN \log \left\|dh_{t\sigma}(z) \left|E^c_{t\sigma}(z) \right. \right\| d\vol_\cN(z)>0,
\eeq
where $E^c_{t\sigma}$ is the one-dimensional central bundle of $h_{t\sigma}$.
\end{enumerate} 
\end{prop}

\begin{proof}  
Fix a sufficiently small $t\in (0, 1]$ such that 
$\ds
\|\varphi^t-\id\|_{C^r}<\delta/2
$,
and
let $\Delta_t$ be the cylindrical neighborhood satisfying \eqref{def Delta t}. 
There exists $\sigma_{t0}\in (0, 1]$ such that  
$\ds
\|h_{t\sigma} - \varphi^t\|_{C^r}<\min\left\{ \delta/2, \delta(t, \Delta_t) \right\}
$
for any $\sigma\in (0, \sigma_{t0}]$,
where $\ds \delta(t, \Delta_t)$ is given by Proposition~\ref{prop:  gentle 0}. 
It follows that  
Statements (1)(2)(3) of Proposition~\ref{prop: property of h}
hold for $h_{t\sigma}$ 
with any $\sigma\in (0, \sigma_{t0}]$.
Furthermore, since $h_{t\si}$ is volume preserving, 
by the Birkhoff ergodic theorem and the Hopf argument,
Statement (3) immediately implies that  $h_{t\si}$ is ergodic with respect to $\vol_\cN$.
 
To show Statement (4),
we follow the arguments by Dolgopyat, Hu and Pesin in~\cite{DHP}, 
which is an elaboration of perturbation technique by
Shub and Wilkinson in~\cite{SW}. 
Due to our special construction of $h_{t\sigma}$ in \eqref{def h t si}, 
its tangent map $d h_{t\sigma}$ preserves $E$ and 
$\ds
\det\left( dh_{t\sigma}(z)|E(z)\right)
=\det\left( dh_{t0}(z)|E(z)\right)
=\det\left( d\varphi^t(z)|E(z)\right)
$ for any $z\in \cN_0$,
where $E$ is any $d\varphi^t$-invariant bundle  
containing the two-dimensional  smooth bundle
$\cE=\cE^u\oplus \cE^c$ given by \eqref{def cE} (including $\cE$ itself).
It follows that  
\eqref{def neg central LE} is equivalent to the inequality
\beq
\label{def L dec}
L_\si<L_0 \ \ \text{for some} \ \sigma\in (0, \sigma_{t0}],
\eeq
in which
\beqn
L_\si:=\int_{\cN} \log \left\|dh_{t\sigma}(z)\left| \cE^u_{t\sigma}(z) \right.\right\| d\vol_\cN(z),
\eeqn
where
$\ds
\cE^u_{t\sigma}
$
is the strong unstable bundle of $h_{t\si}$ inside $\cE$.
We denote by $H_{t\si}$ the first return map of $h_{t\si}$ on $\Delta_t$,
which is defined for $\vol_{\Delta_t}$-almost every $z\in \Delta_t$.
Due to the special form~\eqref{def h t si},
the first return time of $z\in \Delta_t$ under $h_{t\si}$ and $h_{t0}=\varphi^t$ 
are the same, and thus $H_{t\si}(z)=\phi_\si \circ H_{t0}(z)$.
It follows from the ergodicity of $h_{t\si}$ and the Kac formula that
\beq
\label{def L si}
L_\si
=\int_{\Delta_t} \log \eta(\si, z) \, d\vol_{\Delta_t}(z),
\ \text{where} \ \
\eta(\si, z):=\left\|dH_{t\sigma}(z)\left| \cE^u_{t\sigma}(z)\right. \right\|.
\eeq

To prove~\eqref{def L dec}, we proceed with similar
calculations as in  Section 0.5 of~\cite{DHP}.
Under the local coordinate system $(\xi^u, \xi^c, \zeta)$
on $\Delta_t$ given by \eqref{def Cart coord}, 
we assume that the one dimensional space 
$\ds
\cE^{u}_{t\sigma}(z)
$
is spanned by the vector $v(\si, z)=(1, \beta(\si, z), 0)^T\in \cE(z)$
for some continuous function $\beta(\si, \cdot)$ on $\Delta_t$.
Along the invariant bundle
$\cE=\cE^{u}\oplus \cE^c$,
the tangent maps $dH_{t0}$ and $d\phi_\si$ can be written in the matrix form
\beqn
dH_{t0}(z)|\cE(z)=\begin{pmatrix} \eta(z)  & 0 \\ 0 & 1 \end{pmatrix}
\ \ \text{and} \ \ 
d\phi_\si(z)|\cE(z)=\begin{pmatrix} A(\si, z)  & B(\si, z) \\ C(\si, z) & D(\si, z) \end{pmatrix}
\eeqn
for any $z\in \Delta_t$,
where
$\ds \eta(z):=\eta(0, z)=\|dH_{t0}(z)|\cE^u(z)\|$.
Then we get
\beqn
\begin{split}
dH_{t\si}(z)|\cE(z)
&=
d\phi_\si(H_{t0}(z))|\cE(H_{t0}(z))
\, dH_{t0}(z)|\cE(z)   =
\begin{pmatrix} 
\eta(z) A(\si, H_{t0}(z))  &   B(\si, H_{t0}(z)) \\ 
\eta(z) C(\si, H_{t0}(z))  &   D(\si, H_{t0}(z))
\end{pmatrix},
\end{split}
\eeqn
and the invariance equation 
$
\ds dH_{t\si}(z) v(\si, z)= 
\eta(\si, z) v(\si, H_{t\si}(z))
$
yields
\beq
\label{inv eq}
\begin{pmatrix} 
\eta(z) A(\si, H_{t0}(z))  &   B(\si, H_{t0}(z)) \\ 
\eta(z) C(\si, H_{t0}(z))  &   D(\si, H_{t0}(z))
\end{pmatrix}
\begin{pmatrix} 
1 \\ 
\beta(\si, z)
\end{pmatrix}
=
\eta(\si, z)
\begin{pmatrix} 
1 \\ 
\beta(\si, H_{t\si}(z))
\end{pmatrix}.
\eeq
Applying the arguments in the proof of Lemma 0.8 of \cite{DHP},
we obtain
\beqn
\log \eta(\si, z)
=\log \eta(z) - \log \left( D(\si, H_{t0}(z)) - B(\si, H_{t0}(z)) \beta(\si, H_{t\si}(z))\right),
\eeqn
and thus
\beq
L_\si=\int_{\Delta_t}\log \eta(z) d\vol_{\Delta_t}(z)
-\int_{\Delta_t} \log \left( D(\si, z) - B(\si, z) \beta(\si, \phi_\si(z))\right) d\vol_{\Delta_t}(z),
\eeq
where 
we switch $H_{t0}(z)$ to $z$ 
in the second integral, 
using that 
$H_{t\si}=\phi_\si \circ H_{t0}$
and that
$H_{t0}$ preserves $\vol_{\Delta_t}$.
Since
$\ds\beta(0, \cdot)= 0$,
$\ds
A(0, z)=D(0, z)=1,  
B(0, z)=C(0, z)=0,  
$
and 
$\ds
D'_\si(0, z)=\varrho \psi'(\varrho)\psi_1(\|\zeta\|) \sin\theta \cos\theta
$ (see Equation (0.9) of \cite{DHP}),
we obtain
\beq
\label{derivative 1}
\begin{split}
\left. \frac{d L_\si}{d \si}\right|_{\si=0}
&=-\int_{\Delta_t} D'_\si(0, z) d\vol_{\Delta_t}(z) \\
&= -\iint \varrho^2 \psi'(\varrho)\psi_1(\|\zeta\|)  d\varrho d\zeta  
\int_0^{2\pi} \sin\theta \cos\theta d\theta =0.
\end{split}
\eeq
Furthermore, a direct calculation shows that
\beq
\label{derivative 2}
\begin{split}
\left. \frac{d^2 L_\si}{d \si^2}\right|_{\si=0}
= &\int_{\Delta_t} \left[ \left(D'_\si(0, z)\right)^2 - D''_{\si\si}(0, z) \right] d\vol_{\Delta_t}(z) \\
&  + \int_{\Delta_t}  2 B'_\si(0, z) 
 \left. \frac{\p (\beta(\si, \phi_\si(z)))}{\p \si}\right|_{\si=0} d\vol_{\Delta_t}(z) 
\end{split}
\eeq
Again by the invariance equation~\eqref{inv eq}, with $z$ replaced by $H_{t0}^{-1}(z)$, we get
\beqn
\beta(\si, \phi_{\si}(z)) =
\frac{\eta(H_{t0}^{-1}(z)) C(\si,  z) + D(\si, z) \beta(\si, H_{t0}^{-1}(z))}
{\eta(H_{t0}^{-1}(z)) A(\si, z) + B(\si, z) \beta(\si, H_{t0}^{-1}(z))}.
\eeqn
Applying the arguments in the proof of Lemma 0.10 of \cite{DHP},
we obtain
\beqn
\left. \frac{\p (\beta(\si, \phi_{\si}(z)) )}{\p \si}\right|_{\si=0}
=C'_\si(0, z) + 
\sum_{n=1}^\infty \frac{C'_\si(0, H_{t0}^{-n}(z))}{\prod_{k=1}^n \eta(H_{t0}^{-k}(z))}.
\eeqn
Then we can rewrite \eqref{derivative 2} as
$\ds
\left. \frac{d^2 L_\si}{d \si^2}\right|_{\si=0} = J_1 + J_2,
$
where
\beqn
\begin{split}
J_1 := &\int_{\Delta_t} \left[ \left(D'_\si(0, z)\right)^2 - D''_{\si\si}(0, z) + 2 B'_\si(0, z)  C'_\si(0, z) \right] d\vol_{\Delta_t}(z), \\
J_2 := &  \sum_{n=1}^\infty \int_{\Delta_t}  
\frac{2B'_\si(0, z) C'_\si(0, H_{t0}^{-n}(z))}{\prod_{k=1}^n \eta(H_{t0}^{-k}(z))} d\vol_{\Delta_t}(z).
\end{split}
\eeqn
Note that by
\eqref{def Delta t}, \eqref{choose Nt} and \eqref{def expansion rate}, we have
$\ds
\prod_{k=1}^n \eta(H_{t0}^{-k}(z))\ge \eta^{nt N_t}>100^n
$
 for almost every $z\in \Delta_t$.
Following the computation from Equations (0.14)-(0.21) in \cite{DHP}, we get
\beqn
\begin{split}
J_1 &\le -(1-\gamma_{t0}) \int_{\Delta_t} \wpsi^2 d\vol_{\Delta_t} -\frac18  \int_{\Delta_t} \varrho^2 \wpsi_{\varrho}^2 d\vol_{\Delta_t} \\
J_2 & \le 4\left( \int_{\Delta_t} \wpsi^2 d\vol_{\Delta_t} +\int_{\Delta_t} \varrho^2 \wpsi_{\varrho}^2 d\vol_{\Delta_t} \right) 
\cdot \sum_{n=1}^\infty 100^{-n}.
\end{split}
\eeqn
Recall that $\gamma_{t0}<0.1$, we immediately get
\beq
\label{derivative 2'}
\left. \frac{d^2 L_\si}{d \si^2}\right|_{\si=0} = J_1 + J_2
<-0.025\left( \int_{\Delta_t} \wpsi^2 d\vol_{\Delta_t} +\int_{\Delta_t} \varrho^2 \wpsi_{\varrho}^2 d\vol_{\Delta_t} \right) <0.
\eeq
It then follows from \eqref{derivative 1} and \eqref{derivative 2'} that  
$\ds 
L_\si<L_0
$
for sufficiently small $\sigma\in (0, \sigma_{t0}]$,
that is, \eqref{def L dec} holds. 
The proof of this proposition is now complete.
\end{proof}

\begin{remark}
\label{rem: dim34}
We stress that Proposition~\ref{prop: property of h}
heavily relies on the two-dimensional $C^\infty$ smooth unstable-center
bundle $\cE$ of the form \eqref{def cE}, 
which does not exist when $\dim \cN=3$ or $4$.
This smooth bundle allows us to create positive average central Lyapunov exponent
under a small $C^r$ perturbation for an arbitrary $r\in \NN$.
It is worth pointing out that H\"older continuous bundle 
would work for $C^1$ perturbation, but not for $C^r$ perturbation with $r\ge 2$,
since an extra error term 
due to non-smoothness of the bundle might occur to deteriorate 
the estimation of average central Lyapunov exponent (see e.g. Lemma 4.3 in~\cite{DP}). 
\end{remark}

\subsection{Proof of Theorem~\ref{Thmb}}

We first recall a result
by Hu, Pesin and Talitskaya in \cite{HPT13}.

\begin{thm}[Hu-Pesin-Talitskaya~\cite{HPT13}, Theorem 2.3]
\label{thm: HPT1}
Let $f$ be a $C^2$ smooth volume preserving diffeomorphism which is
pointwise partially hyperbolic on an open set $\cO$ of a compact manifold. 
If Properties~\ref{P1}-~\ref{P4} hold for $f:\cO\to \cO$,
and 
\begin{propertyx}
\label{P5}
there is an invariant subset $\cO_1\subset \cO$ of positive volume 
such that the Lyapunov exponent 
of $f$ is positive
at any $z\in \cO_1$ and along any $v\in E^c_f(z)$.
\end{propertyx}

\noindent
Then $f$ is nonuniformly hyperbolic on $\cO$,
and $f|\cO$ is ergodic (in fact, Bernoulli).
\end{thm}

Apparently, the positivity of average central Lyapunov exponent
(see ~\eqref{def neg central LE}) implies Property~\ref{P5}.
Let $h$ be a special gentle perturbation 
$h_{t\si}$, 
which is of the form~\eqref{def h t si}
for sufficiently small $t$ and $\si$,
then the following proposition is 
a direct consequence of 
Theorem~\ref{thm: HPT},
Proposition~\ref{prop: property of h}
and Theorem~\ref{thm: HPT1}.

\begin{prop}
\label{prop h final}
Let $m\ge 5$. 
For any $\delta>0$, any $r\in \NN$ and any admissible sequence $\brho$,
there exists a $C^\infty$ diffeomorphism $h: \cN\to \cN$ 
with the following properties:
\begin{enumerate} 
\item  
$h$ is volume preserving, i.e.,
$h$ preserves the volume $\vol_{\cN}$; 
\item  
$h$ is nonuniformly hyperbolic, i.e.,
$h$ has non-zero Lyapunov exponents at $\vol_{\cN}$-almost every point $z\in \cN$; 
\item $h$ is ergodic (in fact,  Bernoulli) with respect to $\vol_{\cN}$;
\item 
\begin{enumerate}
\item[(i)] 
$\|h-\id\|_{C^r}\le \delta$, i.e.,
$h$ is $\delta$-close to 
the identity in the $C^r$ topology;
\item[(ii)]  $h\in \Diff^\infty_{\brho}(\cN)$, i.e.,
$h$ is $\brho$-flat on $\cN$.
\end{enumerate}
\end{enumerate} 
\end{prop}

Notice that 
Proposition~\ref{prop h final} is parallel to Theorem~\ref{Thmb},
with the only difference on the underlying manifolds. 
Nevertheless, it was shown by Brin~\cite{B}  that
$\cL$ can be embedded into $\RR^{m-1}\times \TT$ with 
trivial normal bundle,
where $\cL$ is the mapping torus  defined in \eqref{def cL}.
Therefore, there exists a continuous mapping 
$\wPsi=\wPsi_{\cN\to \cQ^m}$ from $\cN=\cD^2\times \cL$
to the $m$-dimensional cube $\cQ^m$, 
such that $\wPsi$ satisfies the following properties:
\begin{enumerate}
\item[(a$'$)] the restriction $\ds \wPsi|_{\Int(\cN)}$ is a $C^\infty$ diffeomorphic embedding;
\item[(b$'$)]  $\ds \wPsi(\cN)=\cQ^m$ and $\ds \vol_{\cQ^m}\left(\cQ^m\backslash \wPsi\left( \Int(\cN)\right)\right)=0$;
\item[(c$'$)]  $\ds \wPsi$ is volume preserving, i.e., $\wPsi_* \vol_{\cN}=\vol_{\cQ^m}$.
\end{enumerate}
Note that the above properties are similar to 
those in Statement (1) of Proposition~\ref{prop: Katok top}.
It follows that for any 
$\ve>0$, any $r\in \NN$ and any admissible sequence $\rho$,
there exist $\delta=\delta(\ve, r, \rho)>0$ 
and an admissible sequence $\brho=\brho(\ve, r, \rho)$,
such that if $h$ is 
a  diffeomorphism  obtained from Proposition~\ref{prop h final},
then the map $f=f_{h, \wPsi}$ defined by
\beqn
f(x)=f_{h, \wPsi}(x):=
\begin{cases}
\wPsi\circ h\circ \wPsi^{-1}(x), & \text{if}\  x\in \wPsi\left( \Int(\cN)\right) \\
x, & \text{otherwise}
\end{cases}
\eeqn
is a $C^\infty$ diffeomorphism of $\cQ^m$,
which satisfies 
all the properties listed in Theorem~\ref{Thmb}. 
The proof of Theorem~\ref{Thmb} is now complete.
 
\medskip

\bibliography{NHwCEbib}{}
\bibliographystyle{plain}

\end{document}